\documentclass[12pt,reqno]{amsart}

\usepackage{amssymb}
\usepackage{verbatim}

\newcommand{\lam}{\lambda}
\renewcommand{\phi}{\varphi}

\newcommand{\zet}{\zeta}

\DeclareMathOperator{\med}{mid}

\newcommand{\est}{\varnothing}

\newcommand{\longc}{,\ldots,}

\newcommand{\New}{{\mathcal N}}
\newcommand{\Old}{{\mathcal O}}

\newcommand{\cT}{{\mathfrak T}}
\newcommand{\cR}{{\mathfrak R}}

\newtheorem{lemma}{Lemma}
\newtheorem{claim}{Claim}
\newtheorem{theorem}{Theorem}
\newtheorem{corollary}{Corollary}

\newtheorem{proposition}{Proposition}

\newcommand{\refl}[1]{~\ref{l:#1}}
\newcommand{\refm}[1]{~\ref{m:#1}}
\newcommand{\reft}[1]{~\ref{t:#1}}
\newcommand{\refc}[1]{~\ref{c:#1}}
\newcommand{\refp}[1]{~\ref{p:#1}}
\newcommand{\refs}[1]{~\ref{s:#1}}

\newcommand{\refb}[1]{~\cite{b:#1}}
\newcommand{\refe}[1]{~\eqref{e:#1}}

\theoremstyle{remark}

\newtheorem*{example}{Example}

\author{Vsevolod F. Lev}
\email{seva@math.haifa.ac.il}
\address{Department of Mathematics, The University of Haifa at Oranim,
  Tivon 36006, Israel}

\author{Rom Pinchasi}
\email{room@math.technion.ac.il}
\address{Department of Mathematics, Technion - Israel Institute of
  Technology, Haifa 32000, Israel}

\title[$a\pm b=2c$]{Solving $a\pm b=2c$ \\ in the elements of finite sets}

\begin{document}
\baselineskip=16pt

\begin{abstract}
We show that if $A$ and $B$ are finite sets of real numbers, then the number
of triples $(a,b,c)\in A\times B\times (A\cup B)$ with $a+b=2c$ is at most
$(0.15+o(1))(|A|+|B|)^2$ as $|A|+|B|\to\infty$. As a corollary, if $A$ is
antisymmetric (that is, $A\cap(-A)=\est$), then there are at most
$(0.3+o(1))|A|^2$ triples $(a,b,c)$ with $a,b,c\in A$ and $a-b=2c$. In the
general case where $A$ is not necessarily antisymmetric, we show that the
number of triples $(a,b,c)$ with $a,b,c\in A$ and $a-b=2c$ is at most
$(0.5+o(1))|A|^2$. These estimates are sharp.
\end{abstract}

\maketitle

\section{Introduction and summary of results}

For a finite real set $A$ of given size, the number of three-term arithmetic
progressions in $A$ is maximized when $A$ itself is an arithmetic
progression. This follows by observing that for any integer $1\le k\le |A|$,
the number of three-term progressions in $A$ with the middle term at the
$k$th largest element of $A$ is at most $\min\{k-1,|A|-k\}$. A simple
computation leads to the conclusion that the number of triples $(a,b,c)\in
A\times A\times A$ with $a+b=2c$ is at most $0.5|A|^2+0.5$.

Suppose now that only those progressions with the least element below, and
the greatest element above the median of $A$, are counted; what is the
largest possible number of such ``scattered'' progressions? This problem was
raised in \refb{nppz} in connection with a combinatorial geometry question by
Erd\H os. Below we give it a complete solution; indeed, we solve a more
general problem, replacing the sets of all elements below / above the median
with arbitrary finite sets.
\begin{theorem}\label{t:main}
If $A$ and $B$ are finite sets of real numbers, then the number of triples
$(a,b,c)$ with $a\in A,\ b\in B,\ c\in A\cup B$, and $a+b=2c$, is at most
$0.15(|A|+|B|)^2+0.5(|A|+|B|)$.
\end{theorem}

For a subset $A$ of an abelian group, write $-A:=\{-a\colon a\in A\}$. We say
that $A$ is \emph{antisymmetric} if $A\cap(-A)=\est$. Thus, for instance, any
set of positive real numbers is antisymmetric.

For an antisymmetric set $A$, the number of triples $(a,b,c)$ with $a\in A,\
b\in-A,\ c\in A\cup(-A)$, and $a+b=2c$, is twice the number of triples
$(a,b,c)$ with $a,b,c\in A$ and $a-b=2c$. Hence, Theorem \reft{main} yields
\begin{corollary}\label{c:antisymmetric}
If $A$ is a finite antisymmetric set of real numbers, then the number of
triples $(a,b,c)$ with $a,b,c\in A$ and $a-b=2c$ is at most
$0.3|A|^2+0.5|A|$.
\end{corollary}

The following example shows that the coefficient $0.3$ of Corollary
\refc{antisymmetric}, and therefore also the coefficient $0.15$ of Theorem
\reft{main}, is best possible.

\begin{example}\label{x:main}
Fix an integer $m\ge 1$, and let $A$ consist of all positive integers up to
$m$, and all even integers between $m$ and $4m$ (taking all odd integers will
do as well). Assuming for definiteness that $m$ is even, we thus can write
  $$ A = [1,m] \cup \{m+2,m+4\longc 4m \}. $$
Notice, that $A$ contains $m/2$ odd elements and $2m$ even elements, of which
exactly $m$ are divisible by $4$; in particular, $|A|=5m/2$. For every triple
$(a,b,c)\in A\times A\times A$ with $a-b=2c$, we have $a\equiv b\pmod 2$ and
$a>b$. There are $\binom {m/2}2$ such triples with $a$ and $b$ both odd, and
$2\binom m2$ triples with $a$ and $b$ both even and satisfying $a\equiv
b\pmod 4$. Furthermore, it is not difficult to see that there are
$\frac34m^2$ triples with $a$ and $b$ both even and satisfying $a\not\equiv
b\pmod 4$. Thus, the total number of triples under consideration is
  $$ \binom{m/2}2+2\binom m2+\frac34m^2=\frac{15}8\,m^2-\frac54\,m
                                         = \frac3{10}\,|A|^2-\frac12\,|A|, $$
the first summand matching the main term of Corollary \refc{antisymmetric}.
\end{example}

Our second principal result addresses the same equation as Corollary
\refc{antisymmetric}, but in the general situation where the antisymmetry
assumption got dropped.

\begin{theorem}\label{t:principal}
If $A$ is a finite set of real numbers, then the number of triples $(a,b,c)$
with $a,b,c\in A$ and $a-b=2c$ is at most $0.5|A|^2+0.5|A|$.
\end{theorem}

The main term of Theorem \reft{principal} is best possible as it is easily
seen by considering the set $A=[-m,m]$, where $m\ge 1$ is an integer. For
this set, the number of triples $(a,b,c)\in A\times A\times A$ with $a-b=2c$
is equal to the number of pairs $(a,b)\in A\times A$ with $a$ and $b$ of the
same parity, which is $(m+1)^2+m^2=0.5|A|^2+0.5$.

It is a challenging problem to generalize our results and investigate the
equations $a\pm b=\lambda c$, for a fixed real parameter $\lambda>0$. As it
follows from \cite[Theorem 1]{b:l}, the number of solutions of this equation
in the elements of a finite set of given size is maximized when $\lam=1$, and
the set is an arithmetic progression, centered around $0$. It would be
interesting to determine the largest possible number of solutions for every
\emph{fixed} value of $\lam\ne 1$, or at least to estimate the maximum over
all positive $\lam\ne 1$.

We remark that using a standard technique, our results extend readily onto
finite subsets of torsion-free abelian groups. In contrast, extending
Theorems \reft{main} and \reft{principal} onto groups with a non-zero torsion
subgroup, and in particular onto cyclic groups, seems to be a highly
non-trivial problem requiring an approach completely different from that used
in the present paper.

In the next section we prepare the ground for the proofs of Theorems
\reft{main} and \reft{principal}. The theorems are then proved in Sections
\refs{main-proof} and \refs{principal-proof}, respectively.

\section{The proofs: preparations}\label{s:toolbox}

For finite sets $A,B$, and $C$ of real numbers, let
  $$ T(A,B,C) := \big| \{ (a,b,c)\in A\times B\times C\colon a+b=2c \} \big| . $$

We start with a simple lemma allowing us to confine to the integer case.
\begin{lemma}\label{l:int-red}
For any finite sets $A$ and $B$ of real numbers, there exist finite sets $A'$
and $B'$ of \emph{integer} numbers with $|A'|=|A|,\ |B'|=|B|$ such that
$T(A',B',A'\cup B')=T(A,B,A\cup B)$ and $T(A',-A',A')=T(A,-A,A)$.
\end{lemma}

\begin{proof}
By the (weak version of the) standard simultaneous approximation theorem,
there exist arbitrary large integer $q\ge 1$, along with an integer-valued
function $\phi_q$ acting on the union $A\cup(-A)\cup B$, such that
  $$ \Big|c - \frac{\phi_q(c)}q\Big| < \frac1{4q},\quad c\in A\cup(-A)\cup B. $$
Let $A':=\phi_q(A)$ and $B':=\phi_q(B)$. It is readily verified that if $q$
is large enough, then $|A'|=|A|$ and $|B'|=|B|$ and, moreover, an equality of
the form $a\pm b=2c$ with $a,b,c\in A\cup(-A)\cup B$ holds true if and only
if $\phi_q(a)\pm \phi_q(b)=2\phi_q(c)$. The assertion follows.
\end{proof}

Clearly, for finite sets of integers $A,B$, and $C$ with $|C|\ge|A|+|B|$, the
number of triples $(a,b,c)\in A\times B\times C$ satisfying $a+b=c$ can be as
large as $|A||B|$. Our argument relies on the following lemma which improves
this trivial bound in the case where $|C|<|A|+|B|$.

\begin{lemma}\label{l:leqfis}
If $A,B$ and $C$ are finite sets of real numbers with
$\max\{|A|,|B|\}\le|C|\le |A|+|B|$, then the number of triples
 $(a,b,c)\in A\times B\times C$ satisfying $a+b=c$ does not exceed
  $$ |A||B| - \frac14\,(|A|+|B|-|C|)^2 + \frac14. $$
\end{lemma}

\begin{proof}
We use induction on $|A|+|B|-|C|$. The case where $|A|+|B|-|C|\le 1$ is
immediate, and we thus assume that $|A|+|B|-|C|\ge 2$. If either $A$ or $B$
is empty, then the assertion is readily verified. Otherwise, we let
$a_{\min}:=\min A$ and $b_{\max}:=\max B$, and observe that every $c\in C$
has at most one representation as $c=a_{\min}+b$ with $b\in B$, or of the form
$c=a+b_{\max}$ with $a\in A$. Indeed, the same element
$c\in C$ cannot have representations of both kinds simultaneously, unless
they are identical: for, $a_{\min}+b=a+b_{\max}$ yields
$b-a=b_{\max}-a_{\min}$, whence $a=a_{\min}$ and $b=b_{\max}$. This shows
that removing $a_{\min}$ form $A$, and simultaneously $b_{\max}$ from $B$, we
loose at most $|C|$ triples $(a,b,c)\in A\times B\times C$ with $a+b=c$.
Using now the induction hypothesis to estimate the number of such triples
with $a\ne a_{\min}$ and $b\ne b_{\max}$, we conclude that the total number
of triples under consideration is at most
\begin{multline*}
  |C| + (|A|-1)(|B|-1) - \frac14\,(|A|+|B|-2-|C|)^2 + \frac14 \\
                         = |A||B| -  \frac14\,(|A|+|B|-|C|)^2 + \frac14.
\end{multline*}
\end{proof}

We note that Lemma \ref{l:leqfis} can also be deduced from the following
proposition, which is a particular case of \cite[Theorem~1]{b:l}; see
\cite{b:g,b:hl,b:hlp} for earlier, slightly weaker versions.

For a finite set $A$ of real numbers, write
$\med(A):=\frac12\big(\min(A)+\max(A)\big)$.
\begin{proposition}\label{p:leqfis}
Let $A,B$, and $C$ be finite sets of integers. If $A',B'$, and $C'$ are
blocks of consecutive integers such that $\med(C')$ is at most $0.5$ off from
$\med(A')+\med(B')$, and $|A'|=|A|,\ |B'|=|B|,\ |C'|=|C|$, then the number
of triples $(a,b,c)\in A\times B\times C$ with $a+b=c$ does not exceed the
number of triples $(a',b',c')\in A'\times B'\times C'$ with $a'+b'=c'$.
\end{proposition}

Loosely speaking, Proposition \refp{leqfis} says that the number of solutions
of $a+b=c$ in the variables $a\in A,\ b\in B$, and $c\in C$ is maximized when
$A,B$, and $C$ are blocks of consecutive integers, located so that $C$
captures the integers with the largest number of representations as a sum of
an elements from $A$ and an element from $B$. We leave it to the reader to
see how Lemma \ref{l:leqfis} can be derived from Proposition \ref{p:leqfis}.

%

We use Lemma \refl{leqfis} to estimate the quantity $T(A,B,C)$, which is the
number of solutions of $a+b=c'$ with $a\in A,\ b\in B$, and $c'\in\{2c\colon
c\in C\}$. It is also convenient to recast the estimate of the lemma in terms
of the function $G$ which we define as follows: if $(\xi,\eta,\zet)$ is a
non-decreasing rearrangement of the triple $(x,y,z)$ of real numbers, then we
let
  $$ G(x,y,z) := \begin{cases}
                   \xi\eta &\ \text{if}\ \zet\ge \xi+\eta, \\
                   \xi\eta-\frac14\,(\xi+\eta-\zet)^2 &\ \text{if}\ \zet\le \xi+\eta.
                 \end{cases} $$
Thus, for instance, we have $G(9,6,7)=38$, whereas $G(7,14,6)=42$.

\begin{corollary}\label{c:recasting}
If $A,B$ and $C$ are finite sets of integers, then
  $$ T(A,B,C) \le G(|A|,|B|,|C|)+\frac14. $$
\end{corollary}

We close this section with two lemmas used in the proofs of Theorems
\reft{main} and \reft{principal}, respectively.

For real $x$, we let $x_+:=\max\{x,0\}$ and use $x_+^2$ as an abbreviation
for $(x_+)^2$.
\begin{lemma}\label{l:balancingI}
For any real $x,y$, and $z$, we have
  $$ G\Big(\frac{x+y}2,\frac{x+y}2,z\Big)
       = G(x,y,z) + \frac14\,(x-y)^2 - \frac14\,\big(|x-y|-z\big)_+^2. $$
\end{lemma}

\begin{corollary}\label{c:balancingI}
For any real $x,y$, and $z$, we have
  $$ G\Big(\frac{x+y}2,\frac{x+y}2,z\Big) \ge G(x,y,z). $$
\end{corollary}

\begin{lemma}\label{l:x=y}
If $x$ and $z$ are real numbers with $z\le 2x$, then $G(x,x,z)\le
xz-\frac14\,z^2$.
\end{lemma}

To prove Lemma \refl{balancingI} one can assume $x\le y$ (which does not
restrict the generality) and verify the assertion in the four possible
cases $z\le x$, $x\le z\le(x+y)/2$, $(x+y)/2\le z\le y$, and $z\ge y$.
The proof of Lemma \refl{x=y} goes by straightforward investigation of
the two cases $x\le z$ and $x\ge z$. We omit the details.

\section{Proof of Theorem \reft{main}}\label{s:main-proof}

We use induction on $|A|+|B|$.

By Lemma \refl{int-red}, we can assume that $A$ and $B$ are sets of
integers. For $i,j\in\{0,1\}$ let $A_i:=\{a\in A\colon a\equiv i\pmod 2\}$
and $A_{ij}:=\{a\in A\colon a\equiv i+2j\pmod 4\}$, and define $B_i$ and
$B_{ij}$ in a similar way. Also, write $m:=|A|,\ m_i:=|A_i|,\
m_{ij}:=|A_{ij}|,\ n:=|B|,\ n_i:=|B_i|$, and $n_{ij}:=|B_{ij}|$. Applying a
suitable affine transformation to $A$ and $B$, we can assume without loss of
generality that $A\cup B$ contains both even and odd elements, and the total
number of even elements in $A$ and $B$ is at least as large as the total
number of odd elements:
\begin{equation}\label{e:assumptions}
   0<m_1+n_1\le m_0+n_0<m+n.
\end{equation}

Keeping the notation introduced at the beginning of Section \refs{toolbox},
we want to estimate the quantity $T(A,B,A\cup B)$. Observing that $a+b=2c$
implies that $a$ and $b$ are of the same parity, we write
\begin{equation}\label{e:decomposition}
  T(A,B,A\cup B) = T(A_0,B_0,A_0\cup B_0) + T(A_0,B_0,A_1\cup B_1)
                                                  + T(A_1,B_1,A\cup B)
\end{equation}
and estimate separately each of the three summands in the right-hand side.

For the first summand, we notice that $a_0+b_0=2c_0$ with $a_0\in A_0$,
$b_0\in B_0$, and $c_0\in A_0\cup B_0$, implies that $a_0/2$ and $b_0/2$ are
of the same parity. Hence, either $a_0\in A_{00}$ and $b_0\in B_{00}$, or
$a_0\in A_{01}$ and $b_0\in B_{01}$, leading to the upper bound
$m_{00}n_{00}+m_{01}n_{01}$. On the other hand, we can use induction (cf.
\refe{assumptions}) to estimate the first summand by
$0.15(m_0+n_0)^2+0.5(m_0+n_0)$. As a result,
\begin{equation}\label{e:first-summand}
  T(A_0,B_0,A_0\cup B_0)
     \le \min \{ 0.15(m_0+n_0)^2, m_{00}n_{00}+m_{01}n_{01} \} + 0.5(m_0+n_0).
\end{equation}

Similar parity considerations show that if $a_0+b_0=2c_1$ with $a_0\in A_0$,
$b_0\in B_0$, and $c_1\in A_1\cup B_1$, then either $a_0\in A_{00}$ and
$b_0\in B_{01}$, or $a_0\in A_{01}$ and $b_0\in B_{00}$. Therefore, using
Corollary \refc{recasting}, we get
\begin{align}\label{e:second-summand}
  T(A_0,B_0,A_1\cup B_1)
    &= T(A_{00},B_{01},A_1\cup B_1) + T(A_{01},B_{00},A_1\cup B_1) \notag \\
    &\le G(m_{00},n_{01},m_1+n_1)+G(m_{01},n_{00},m_1+n_1) + 0.5.
\end{align}

For the last summand in \refe{decomposition} we use the trivial estimate
\begin{equation}\label{e:third-summand}
  T(A_1,B_1,A\cup B) \le m_1n_1 \le 0.25(m_1+n_1)^2.
\end{equation}
Substituting \refe{first-summand}--\refe{third-summand} into
\refe{decomposition}, we get
\begin{align}\label{e:res-est}
  T(A,B,A\cup B)
    &\le \min \{ 0.15(m_0+n_0)^2, m_{00}n_{00}+m_{01}n_{01} \} \notag \\
    &\qquad + G(m_{00},n_{01},m_1+n_1)+G(m_{01},n_{00},m_1+n_1) \notag \\
    &\qquad + \frac14\,(m_1+n_1)^2 + 0.5(m_0+n_0)+0.5.
\end{align}
Recalling \refe{assumptions}, we estimate the remainder terms as
  $$ 0.5(m_0+n_0) + 0.5 \le 0.5(m+n). $$
To estimate the main term, for real $x_0,x_1,y_0,y_1$ we write
\begin{equation}\label{e:xy}
  s := x_0+x_1+y_0+y_1
\end{equation}
and let
\begin{align}\label{e:f-def}
  f(x_0,x_1,y_0,y_1)
    &:= \min \{ 0.15s^2, x_0y_0+x_1y_1 \} \notag \\
    &\qquad + G(x_0,y_1,1-s) + G(x_1,y_0,1-s) \notag \\
    &\qquad + 0.25(1-s)^2.
\end{align}
Remainder terms dropped, the right-hand side of \refe{res-est} can then be
written as $(m+n)^2f(\xi_0,\xi_1,\eta_0,\eta_1)$, where
  $$ \xi_0:=\frac{m_{00}}{m+n},\ \xi_1:=\frac{m_{01}}{m+n},
                     \ \eta_0:=\frac{n_{00}}{m+n},
                             \ \text{and}\ \eta_1:=\frac{n_{01}}{m+n}. $$

With \refe{assumptions} in mind, we see that to complete the argument it
suffices to prove the following lemma.
\begin{lemma}\label{l:inequality}
For the function $f$ defined by \refe{xy}--\refe{f-def}, we have
  $$ \max \{ f(x_0,x_1,y_0,y_1)
                \colon x_0,x_1,y_0,y_1\ge 0,\ 1/2\le s\le 1 \} \le 0.15. $$
\end{lemma}
The inequality of Lemma \refl{inequality} is surprisingly delicate, and the
proof presented in the remaining part of this section is rather tedious. The
reader trusting us about the proof may wish to skip on to Section
\refs{principal-proof}, where the proof of Theorem \reft{principal}
(independent of Theorem \reft{main}) is given.

\begin{proof}[Proof of Lemma \refl{inequality}]
Since $f(x_0,x_1,y_0,y_1)=f(y_0,y_1,x_0,x_1)$, switching, if necessary, $x_0$
with $y_0$, and $x_1$ with $y_1$, we can assume that
\begin{equation}\label{e:x>y}
  x_0+x_1 \ge y_0+y_1.
\end{equation}
Similarly, $f(x_0,x_1,y_0,y_1)=f(x_1,x_0,y_1,y_0)$ shows that $x_0$ can be
switched with $x_1$, and $y_0$ with $y_1$ to ensure that
\begin{equation}\label{e:0>1}
  x_0+y_0 \ge x_1+y_1.
\end{equation}
(Observe, that switching $x_0$ with $x_1$ and $y_0$ with $y_1$ does not
affect \refe{x>y}.) Thus, from now on we assume that \refe{x>y} and
\refe{0>1} hold true.

Our big plan is to investigate the effect made on $f$ by replacing the
variables $x_0$ and $y_1$ with their average $(x_0+y_1)/2$, and,
simultaneously, replacing the variables $x_1$ and $y_0$ with their average
$(x_1+y_0)/2$. We show that either
\begin{equation}\label{e:balancingOk}
  f\Big(\frac{x_0+y_1}2,\frac{x_1+y_0}2,\frac{x_1+y_0}2,\frac{x_0+y_1}2\Big)
                                                       \ge f(x_0,x_1,y_0,y_1)
\end{equation}
(meaning that $f$ is non-decreasing under such ``balancing''), or
\begin{gather}
   x_0 \ge y_1+(1-s), \label{e:x0large} \\
   y_0 \ge x_1+(1-s), \label{e:y0large}
\intertext{and}
   3(x_0+y_0)+(x_1+y_1) \ge 2. \label{e:312}
\end{gather}
In both cases, the problem reduces to maximizing a function in just two
variables.

We thus assume that \refe{balancingOk} fails, aiming to prove that
\refe{x0large}--\refe{312} hold true. Along with \refe{f-def} and Corollary
\refc{balancingI}, our assumption implies
  $$ \frac12\,(x_0+y_1)(x_1+y_0) < x_0y_0+x_1y_1, $$
simplifying to
  $$ (x_0-y_1)(x_1-y_0) < 0. $$
Writing \refe{0>1} as $x_0-y_1\ge x_1-y_0$, we conclude that
\begin{equation}\label{e:skew}
  x_0>y_1 \ \text{and}\ y_0>x_1
\end{equation}
(which the reader may wish to compare with \refe{x0large} and
\refe{y0large}).

Let
  $$ \Old := x_0y_0+x_1y_1 + G(x_0,y_1,1-s) + G(x_1,y_0,1-s) $$
and
  $$ \New := \frac12\,(x_0+y_1)(x_1+y_0)
       + G\Big(\frac{x_0+y_1}2,\frac{x_0+y_1}2,1-s \Big)
       + G\Big(\frac{x_1+y_0}2,\frac{x_1+y_0}2,1-s \Big) $$
(the script letters standing for ``old'' and ``new''); thus, $\New<\Old$ by
the assumption that \refe{balancingOk} fails, \refe{f-def}, and Corollary
\refc{balancingI}. From Lemma \refl{balancingI} and \refe{skew} we get
\begin{align*}
  \New-\Old
    &= \frac12\,(x_0-y_1)(x_1-y_0)
          + \frac14\,(x_0-y_1)^2 - \frac14\,(|x_0-y_1|-(1-s))_+^2 \\
    &\phantom{= \frac12\,(x_0-y_1)(x_1-y_0)\ }
          + \frac14\,(x_1-y_0)^2 - \frac14\,(|x_1-y_0|-(1-s))_+^2 \\
    &= \frac14\,(x_0+x_1-y_0-y_1)^2  - \frac14\,(x_0-y_1-(1-s))_+^2 \\
    &\phantom{= \frac14\,(x_0+x_1-y_0-y_1)^2\ }
          - \frac14\,(y_0-x_1-(1-s))_+^2.
\end{align*}
Analyzing the expression in the right-hand side we see that if \refe{y0large}
were false, then $\New<\Old$ along with \refe{x>y} would give
  $$ x_0+x_1-y_0-y_1 < x_0-y_1-(1-s), $$
which is \refe{y0large} in disguise. This contradiction shows that
\refe{y0large} is true. We now readily get \refe{x0large} as a consequence of
\refe{y0large} and \refe{x>y}, and \refe{312} is just a sum of \refe{y0large}
and \refe{x0large}.

To summarize, there are two major cases to consider: that where
\refe{balancingOk} holds true, and that where \refe{x0large}--\refe{312} hold
true. Since in the second case we have $G(x_0,y_1,1-s)=y_1(1-s)$ and
$G(x_1,y_0,1-s)=x_1(1-s)$, the proof of Lemma \refl{inequality} will be
complete once we establish the following claims.

\begin{claim}\label{m:claim1}
We have $f(x_0,x_1,x_1,x_0)\le 0.15$ for any $x_0,x_1\ge 0$ with
$s:=2(x_0+x_1)\in[1/2,1]$.
\end{claim}

\begin{claim}\label{m:claim2}
For real $x_0,x_1,y_0$, and $y_1$, write $s:=x_0+x_1+y_0+y_1$ and let
  $$ g(x_0,x_1,y_0,y_1) =
               \min\{0.15s^2,x_0y_0+x_1y_1\}+(x_1+y_1)(1-s)+0.25(1-s)^2. $$
Then $g(x_0,x_1,y_0,y_1)\le 0.15$ whenever $x_0,x_1,y_0,y_1\ge 0$ satisfy
\refe{312}, and $s\le 1$.
\end{claim}

\begin{proof}[Proof of Claim \refm{claim1}]
As
\begin{multline*}
  f(x_0,x_1,x_1,x_0) = \min\{0.15s^2, 2x_0x_1\} \\
                             + G(x_0,x_0,1-s) + G(x_1,x_1,1-s) + 0.25(1-s)^2,
\end{multline*}
and since $x_0+x_1=\frac12s$ implies $2x_0x_1\le\frac18\,s^2<0.15s^2$, we
have to show that
\begin{equation}\label{e:rom1}
  2x_0x_1 + G(x_0,x_0,1-s) + G(x_1,x_1,1-s) + 0.25(1-s)^2 \le 0.15.
\end{equation}

We distinguish three cases.

\medskip
\paragraph{Case I: $\max\{x_0,x_1\}\le\frac12(1-s)$}

In this case, from the definition of the function $G$, we have
$G(x_0,x_0,1-s)=x_0^2$ and $G(x_1,x_1,1-s)=x_1^2$. Therefore, \refe{rom1}
reduces to
  $$ 2x_0x_1+x_0^2+x_1^2+0.25(1-s)^2 \le 0.15 $$
or, equivalently,
\begin{equation}\label{e:rom2}
  0.25s^2 + 0.25(1-s)^2 \le 0.15.
\end{equation}
To show this we notice that our present assumption
$\max\{x_0,x_1\}\le\frac12(1-s)$ yields $s=2(x_0+x_1)\le 2-2s$, implying
$s\le\frac23$. However, the largest value attained by the left-hand side of
\refe{rom2} in the range $\frac12\le s\le\frac23$ is easily seen to be
$5/36<0.15$.

\medskip
\paragraph{Case II: $\min\{x_0,x_1\}\ge\frac12(1-s)$}

In this case, by Lemma \ref{l:x=y}, we have $G(x_0,x_0,1-s)\le
x_0(1-s)-0.25(1-s)^2$ and $G(x_1,x_1,1-s)\le x_1(1-s)-0.25(1-s)^2$.
Consequently, the left-hand side of \refe{rom1} is at most
\begin{align*}
  2x_0x_1 &+ x_0(1-s)+x_1(1-s)-0.25(1-s)^2 \\
    &\le \frac12\,(x_0+x_1)^2 + (x_0+x_1)(1-s)-0.25(1-s)^2 \\
    &= \frac18\,s^2+\frac12s(1-s)-0.25(1-s)^2 \\
    &= -\frac58\,\left(s-\frac45\right)^2+0.15 \\
    &\le 0.15.
\end{align*}

\medskip
\paragraph{Case III: $x_0\le\frac12(1-s)\le x_1$
  (the case $x_1 \leq \frac{1-s}{2} \le x_0$ being symmetric)}

In this case $G(x_0,x_0,1-s)=x_0^2$, while from Lemma \ref{l:x=y} we have
$G(x_1,x_1,1-s)\le x_1(1-s)-0.25(1-s)^2$; thus, \refe{rom1} reduces to
  $$ 2x_0x_1+x_0^2+x_1(1-s) \le 0.15 $$
and, substituting $x_0=\frac12\,s-x_1$ and re-arranging the terms, to
\begin{equation}\label{e:rom4}
  \frac14\,(2s^2-2s+1) - \left(x_1-\frac12\,(1-s)\right)^2 \le 0.15.
\end{equation}

Observing that $2s^2-2s+1$ is increasing for $s\ge 1/2$ (and recalling that
$s\ge\frac12$ by the assumptions of the claim), we conclude that if
$s\le\frac23$, then the left-hand side or \refe{rom4} does not exceed
  $$ \frac14\,\left(2\cdot\frac49-2\cdot\frac23+1\right)
                                                = \frac5{36} < 0.15. $$

If, on the other hand, $s\ge\frac23$, then we have
  $$ x_1 = \frac12\,s-x_0 \ge \frac12\,s-\frac12\,(1-s)
                                     = s-\frac12 \ge \frac12\,(1-s), $$
whence the left-hand side of \refe{rom4} does not exceed
  $$ \frac14\,(2s^2-2s+1)
            - \left(\left(s-\frac12\right)-\frac12\,(1-s)\right)^2
  = -\frac74\,\left(s-\frac57\right)^2+\frac17 < 0.15. $$
\end{proof}

\begin{proof}[Proof of Claim \refm{claim2}]
Since replacing $x_0$ and $y_0$ with their average $(x_0+y_0)/2$ and,
simultaneously, $x_1$ and $y_1$ with their average $(x_1+y_1)/2$, can only
increase the value of $g$, and does not affect the validity of \refe{312}, we
can assume that $y_0=x_0$ and $y_1=x_1$. Thus, we want to show that in the
region defined by
\begin{equation}\label{e:triangle}
  x_0,x_1\ge 0,\ x_0+x_1\le 1/2,\ \text{and}\ 3x_0+x_1\ge 1,
\end{equation}
we have
  $$ g(x_0, x_1,x_0,x_1) \le 0.15. $$
Observing that
\begin{align*}
  g(x_0, x_1,x_0,x_1)
    &= \min\{0.6(x_0+x_1)^2,x_0^2+x_1^2\}
                                   + 0.25(1-2x_0-2x_1)(1-2x_0+6x_1) \\
    &=  \min\{0.6(x_0+x_1)^2,x_0^2+x_1^2\}
                                       + x_0^2-2x_0x_1-3x_1^2-x_0+x_1+0.25,
\end{align*}
the estimate to prove can be re-written as
\begin{equation*}\label{e:gx0x1}
  \min\{u(x_0,x_1),v(x_0,x_1)\} \le -0.1,
\end{equation*}
where
\begin{align*}
  u(x_0,x_1) &= 2x_0^2-2x_0x_1-2x_1^2-x_0+x_1
\intertext{and}
  v(x_0,x_1) &= 1.6x_0^2-0.8x_0x_1-2.4x_1^2-x_0+x_1.
\end{align*}

Conditions \refe{triangle} determine on the coordinate plane $(x_0,x_1)$ a
triangle with the vertices at $(1/3,0),\ (1/2,0)$, and $(1/4,1/4)$. If
$\phi:=(3-\sqrt5)/2$, then the line $x_1=\phi x_0$ splits this triangle into
two parts: a smaller triangle $\cT$ which inherits the vertex $(1/4,1/4)$ of
the original triangle, and a rectangle $\cR$ inheriting the vertices
$(1/3,0)$ and $(1/2,0)$ of the original triangle. (We consider both $\cT$ and
$\cR$ as closed regions, so that they intersect by a segment.) The reason to
partition the large rectangle as indicated is that
  $$  \min\{u(x_0,x_1),v(x_0,x_1)\} =
        \begin{cases}
          u(x_0,x_1)\ &\text{if}\ (x_0,x_1)\in\cT, \\
          v(x_0,x_1)\ &\text{if}\ (x_0,x_1)\in\cR,
        \end{cases} $$
as one can easily verify; we therefore have to prove that $u(x_0,x_1)\le-0.1$
for all $(x_0,x_1)\in\cT$, and $v(x_0,x_1)\le -0.1$ for all
$(x_0,x_1)\in\cR$.

To this end we observe that, as a simple computation shows, the only critical
point of $u$ is $(0.3,0.1)$, and the only critical point of $v$ is
$(0.35,0.15)$. Since the former point lies on the line $3x_0+x_1=1$, and the
latter on the line $x_0+x_1=1/2$, these points do not belong to the interiors
of $\cT$ and $\cR$. Hence, the maxima of $u$ on $\cT$, and of $v$ on $\cR$,
are attained on the boundary of these regions. To complete the proof we now
observe that

\medskip
\paragraph{I} if $1/3\le x_0\le 1/2$ and $x_1=0$, then
  $$ v(x_0,x_1) = 1.6x_0^2-x_0
              \le 1.6\cdot\frac14 - \frac12 = -0.1 $$ (as $1.6x_0^2-x_0$
is an increasing function of $x_0$ on the interval $[1/3,1/2]$);

\medskip
\paragraph{II} if $x_0+x_1=1/2$, then
  $$ u(x_0,x_1)=x_0^2+x_1^2-0.25\ge 0, $$
and
  $$v(x_0,x_1)=0.6(x_0+x_1)^2-0.25=-0.1. $$

\medskip
\paragraph{III} if $3x_0+x_1=1$, then
  $$ u(x_0,x_1) = -10x_0^2+6x_0-1 = -10(x_0-0.3)^2-0.1 \le -0.1; $$
if, in addition, $(x_0,x_1)\in\cR$, then
  $$ 1 = 3x_0+x_1 \le (3+\phi)x_0, $$
whence $x_0\ge 1/(3+\phi)=(9+\sqrt 5)/38$ and therefore
\begin{align*}
  v(x_0,x_1)
    &= -17.6x_0^2+9.6x_0-1.4 \\
    &\le -17.6\cdot\left(\frac{9+\sqrt5}{38}\right)^2
       + 9.6\cdot\frac{9+\sqrt5}{38} - 1.4 \\
    &= -0.1001\ldots
\end{align*}
(as $(9+\sqrt 5)/38>3/11$, and $-17.6x_0^2+9.6x_0-1.4$ is a decreasing function
of $x_0$ for $x_0\ge 3/11$);

\medskip
\paragraph{IV} if $x_1=\phi x_0$ and $(x_0,x_1)\in\cT\cap\cR$, then
\begin{align*}
  u(x_0,x_1)=v(x_0,x_1)
    &=(2-2\phi-2\phi^2)x_0^2+(\phi-1)x_0 \\
    &= 4(\sqrt5-2)x_0^2-\frac{\sqrt5-1}2\,x_0,
\end{align*}
being a convex function of $x_0$, attains its maximum for a value of $x_0$
which is on the boundary of the triangle $\cT\cup\cR$. However, we have
already seen that $u$ and $v$ do not exceed the value of $-0.1$ on the part
of the boundary they are responsible for.
\end{proof}

This finally completes the proof of Lemma \refl{inequality}, and thus the
whole proof of Theorem \reft{main}.
\end{proof}

\section{Proof of Theorem \reft{principal}}\label{s:principal-proof}

As in the proof of Theorem \reft{main}, we use induction on $|A|$ and, with
Lemma \refl{int-red} in mind, assume that $A$ is a set of integers. Again,
for $i,j\in\{0,1\}$ we let $A_i:=\{a\in A\colon a\equiv i\pmod 2\}$ and
$A_{ij}:=\{a\in A\colon a\equiv i+2j\pmod 4\}$, and write $m:=|A|,\
m_i:=|A_i|$, and $m_{ij}:=|A_{ij}|$. Dividing through all elements of $A$ by
their greatest common divisor and replacing $A$ with $-A$, if necessary, we
can assume that
\begin{equation}\label{e:ass}
  0\le m_0<m \ \text{and}\ m_{00}\le m_{01}.
\end{equation}
We want to show that $T(A,-A,A)\le 0.5m^2+0.5m$.

We distinguish two major cases, depending on which of $m_0$ and $m_1$ is
larger.

\medskip
\paragraph{Case I: $m_0\ge m_1$}
Since $a-b=2c$ implies that $a$ and $b$ are of the same parity, we have the
decomposition
\begin{align*}
  T(A,-A,A)
    &= T(A_1,-A_1,A) + T(A_0,-A_0,A_1) + T(A_0,-A_0,A_0) \notag \\
    &= T(A_1,-A_1,A) + T(A_{00},-A_{01},A_1) + T(A_{01},-A_{00},A_1) \notag \\
    &{\hskip 3in} + T(A_0,-A_0,A_0)
\end{align*}
(for the second equality notice that $a_0-b_0=2c_1$ with $a_0,b_0\in A_0$ and
$c_1\in A_1$ implies that either $a_0\in A_{00},\ b_0\in A_{01}$, or $a_0\in
A_{01},\ b_0\in A_{00}$). We estimate the first summand in the right-hand
side trivially, and use the induction hypothesis (cf.~\refe{ass}) for the
last summand, and Corollary \refc{recasting} for the remaining two summands;
this gives
\begin{equation}\label{e:decomp}
  T(A,-A,A) \le m_1^2 + 2G(m_{00},m_{01},m_1)
                                   + \frac12 + \frac12\,m_0^2 + \frac12\,m_0.
\end{equation}

Keeping in mind \refe{ass}, we now consider three further subcases.

\medskip
\paragraph{Subcase I.a: $\max\{m_{00},m_{01},m_1\}=m_1$}
Using \refe{decomp} and recalling that, by the assumption of Case I, we have
$m_1\le m_0=m_{00}+m_{01}$, we get
\begin{align*}
  T(A,-A,A)
    &\le m_1^2 + 2m_{00}m_{01} - \frac12\,(m_{00}+m_{01}-m_1)^2 + \frac12
                                          + \frac12\,m_0^2 + \frac12\,m_0 \\
    &=   \frac12\,m_1^2 + 2m_{00}m_{01} + m_0m_1 + \frac12\,m_0 + \frac12 \\
    &\le \frac12\,m_1^2 + \frac12\,m_{00}^2+\frac12\,m_{01}^2
                                     + m_{00}m_{01} + m_0m_1 + \frac12\,m \\
    &=   \frac12\,m^2 + \frac12\,m.
\end{align*}

\medskip
\paragraph{Subcase I.b: $\max\{m_{00},m_{01},m_1\}=m_{01}\le m_{00}+m_1$}
By \refe{decomp}, using the estimate $\frac12\,m_0^2\le m_{00}^2+m_{01}^2$,
we obtain
\begin{align*}
  T(A,-A,A)
    &\le m_1^2 + 2m_{00}m_1 - \frac12\,(m_{00}+m_1-m_{01})^2 + \frac12
                                          + \frac12\,m_0^2 + \frac12\,m_0 \\
    &\le \frac12\,m_1^2 + m_{00}m_1 + \frac12\,m_{00}^2 + \frac12\,m_{01}^2
            + m_{00}m_{01} + m_1m_{01} + \frac12+\frac12\,m_0 \\
    &\le \frac12\,m^2 + \frac12 + \frac12\,m_0 \\
    &\le \frac12\,m^2 + \frac12\,m.
\end{align*}

\medskip
\paragraph{Subcase I.c: $\max\{m_{00},m_{01},m_1\}=m_{01}\ge m_{00}+m_1$}
In this case we have $G(m_{00},m_{01},m_1)\le m_{00}m_1$, and \refe{decomp}
along with $\frac12\,m_1<m_1\le m_{01}-m_{00}$ give
\begin{align*}
  T(A,-A,A)
    &\le m_1^2 + 2m_{00}m_1 + \frac12 + \frac12\,m_0^2 + \frac12\,m_0 \\
    &= \frac12\,(m_1+m_0)^2 + \frac12\,m_1^2 - m_0m_1+2m_{00}m_1
                                             + \frac12 + \frac12\,m_0 \\
    &= \frac12\,m^2 + m_1\left(\frac12\,m_1+m_{00}-m_{01}\right)
                                             + \frac12 + \frac12\,m_0 \\
    &< \frac12\,m^2 + \frac12\,m.
\end{align*}

\medskip
\paragraph{Case II: $m_1\ge m_0$}
In this case we use the decomposition
\begin{align*}
  T(A,-A,A)
    &=   T(A_0,-A_0,A) + T(A_1,-A_1,A_0) + T(A_1,-A_1,A_1) \\
    &=   T(A_0,-A_0,A) + T(A_{10},-A_{10},A_0) + T(A_{11},-A_{11},A_0) \\
    &{\hskip 3in} + T(A_1,-A_1,A_1). \\
\end{align*}

Using the trivial bound $m_0^2$ for the first summand, applying Corollary
\refc{recasting} to estimate the second and third summands, and observing
that $a_1-b_1=2c_1\ (a_1,b_1,c_1\in A_1)$ implies that exactly one of $a_1$
and $a_2$ is in $A_{10}$ and another is in $A_{11}$, we get
\begin{equation}\label{e:CaseII}
  T(A,-A,A) \le m_0^2 + G(m_{10},m_{10},m_0) + G(m_{11},m_{11},m_0)
                                                   + \frac12 + 2m_{10}m_{11}.
\end{equation}
Since the right-hand side is symmetric in $m_{10}$ and $m_{11}$, without loss
of generality we assume that $m_{10}\le m_{11}$. Consequently, by the
assumption of Case II, we have $m_0\le m_1\le 2m_{11}$, and to complete the
proof we consider two subcases, according to whether the stronger estimate
$m_0\le 2m_{10}$ holds.

\medskip
\paragraph{Subcase II.a: $m_0\le 2m_{10}$}
In this case, by \refe{CaseII} and Lemma \refl{x=y}, and in view of
$2m_{10}m_{11}\le\frac12(m_{10}+m_{11})^2=\frac12\,m_1^2$, we have
\begin{align*}
  T(A,-A,A)
    &\le m_0^2 + \left( m_{10}m_0 - \frac14\,m_0^2 \right)
               + \left( m_{11}m_0 -\frac14\,m_0^2 \right)
                                 + \frac12 + \frac12\,m_1^2 \\
    &=   \frac12\,m_0^2 + m_0m_1 + \frac12 m_1^2 + \frac 12 \\
    &=   \frac12\,m^2 + \frac12.
\end{align*}

\medskip
\paragraph{Subcase II.b: $2m_{10}\le m_0\le 2m_{11}$}
Acting as in the previous subcase, but using the trivial estimate for the
second summand in \refe{CaseII}, we get
\begin{align*}
  T(A,-A,A)
    &\le m_0^2 + m_{10}^2 + \left( m_{11}m_0 - \frac14\,m_0^2 \right)
                                 + \frac12 + 2m_{10}m_{11} \\
    &\le \frac34\,m_0^2 + m_{10}^2 + m_{11}m_0
                 + \frac32\,m_{10}m_{11} + \frac14\,m_{10}^2
                                                + \frac14\,m_{11}^2 + \frac12 \\
    &=  \frac12(m_0+m_{10}+m_{11})^2
                 - \frac14\,(m_{01}+m_{11}-m_0)(m_0+m_{11}-3m_{10}) + \frac12 \\
    &\le \frac12\,m^2 + \frac12,
\end{align*}
the last inequality following from $m_{01}+m_{11}-m_0\ge 0$ and
$m_0+m_{11}-3m_{10}\ge 0$, by the present subcase assumptions.


This completes the proof of Theorem \reft{principal}.


\end{document}